\documentclass[12pt, reqno]{amsart} 
\usepackage{amsmath, amsthm, amscd, amsfonts, amssymb, graphicx, color}
\usepackage[bookmarksnumbered, colorlinks, plainpages]{hyperref}
\usepackage{tikz}

\tikzstyle{vertex}=[circle, draw, inner sep=0pt, minimum size=8pt] 
\newcommand{\vertex}{\node[vertex]}
\newtheorem{theorem}{Theorem}[section]
\newtheorem{lemma}[theorem]{Lemma}
\newtheorem{proposition}[theorem]{Proposition}
\newtheorem{corollary}[theorem]{Corollary}
\theoremstyle{definition}

\theoremstyle{remark}
\newtheorem{remark}[theorem]{Remark}
\numberwithin{equation}{section}

\begin{document}

\title[Zero-Annihilator Graphs of Commutative rings]{Zero-Annihilator Graphs of Commutative rings}
\author[Mostafanasab]{Hojjat Mostafanasab}

%\thanks{$^*$Corresponding author}

\begin{abstract}
Assume that $R$ is a commutative ring with nonzero identity.
In this paper, we introduce and investigate {\it zero-annihilator graph of} $R$
denoted by $\mathtt{ZA}(R)$. It is the graph whose vertex set is the set of all nonzero nonunit elements of
$R$ and two distinct vertices $x$ and $y$ are adjacent whenever $${\rm Ann}_R(x)\cap {\rm Ann}_R(y)=\{0\}.$$
\end{abstract}

\maketitle

\section{\protect\bigskip Introduction}

Throughout this paper all rings are commutative with nonzero identity.
In \cite{Bec}, Beck associated to a ring $R$ its zero-divisor graph $G(R)$ whose vertices are the zero-divisors
of $R$ (including 0), and two distinct vertices $x$ and $y$ are adjacent if $xy=0$. Later, in \cite{AL}, Anderson and Livingston
studied the subgraph $\Gamma(R)$ (of $G(R)$) whose vertices are the nonzero zero-divisors of $R$. In the recent years, several researchers have done interesting and enormous works on this field of study. For instance, see \cite{AN,Bad,Red}. 
The concept of co-annihilating ideal graph of a ring $R$, denoted by $\mathcal{A}_R$ was introduced by Akbari et al.
in \cite{AAAS}. As in \cite{AAAS}, {\it co-annihilating ideal graph of} $R$,
denoted by  $\mathcal{A}_R$,  is a graph whose vertex set is the set of
all non-zero proper ideals of $R$ and two distinct vertices $I$ and $J$ are adjacent whenever
${\rm Ann}_R(I)\cap {\rm Ann}_R(J)=\{0\}$. 
In the present paper, we introduce {\it zero-annihilator graph of} $R$
denoted by $\mathtt{ZA}(R)$. It is the graph whose vertex set is the set of all nonzero nonunit elements of
$R$ and two distinct vertices $x$ and $y$ are adjacent whenever ${\rm Ann}_R(Rx+Ry)={\rm Ann}_R(x)\cap {\rm Ann}_R(y)=\{0\}$. 

Let $G$ be a simple graph with the vertex set ${\rm V}(G)$ and edge set ${\rm E}(G)$. For every vertex
$v \in {\rm V}(G)$, ${\rm N}_G(v)$ is the set $\{u\in {\rm V}(G)\mid uv\in {\rm E}(G)\}$. The {\it degree} of a vertex $v$ is defined
as ${\rm deg}_G(v)=|{\rm N}_G(v)|$. The {\it minimum degree} of $G$ is denoted by $\delta(G)$. 
Recall that a graph $G$ is {\it connected} if there is a path between every
two distinct vertices. For distinct vertices $x$ and $y$ of a connected graph $G$, let ${\rm d}_G(x, y)$ be
the length of the shortest path from $x$ to $y$. The {\it diameter} of a connected graph $G$ is ${\rm diam}(G) = {\rm sup}\{{\rm d}_G(x, y) \mid
x \mbox{ and } y \mbox{ are distinct vertices of } G\}$. The {\it girth} of $G$, denoted by ${\rm girth}(G)$, is
defined as the length of the shortest cycle in $G$ and ${\rm girth}(G) = \infty$ if $G$ contains
no cycles. 
A {\it bipartite graph} is a graph all of whose vertices can be partitioned into two parts $U$ and $V$ such that
every edge joins a vertex in $U$ to a vertex in $V$. 
A {\it complete bipartite graph} is a bipartite graph with parts $U,V$ such that every vertex in $U$ is adjacent to every vertex in $V$.
A graph in which all vertices have degree $k$ is called a $k$-{\it regular graph}. 
A graph in which each pair of distinct vertices is joined by an
edge is called a {\it complete graph}. Also, if a graph $G$ contains one vertex to
which all other vertices are joined and $G$ has no other edges, is called a {\it star
graph}. A {\it clique} in a graph $G$ is a subset of pairwise adjacent vertices and
the number of vertices in a maximum clique of $G$, denoted
by $\omega(G)$, is called the {\it clique number} of $G$. 
Obviously, $\chi(G)\geq\omega(G)$.

\bigskip

\section{Some properties of $\mathtt{ZA}(R)$}

Recall that, an {\it empty graph} is a graph with no edges. A {\it B\'{e}zout ring} is a ring in which all finitely generated ideals are principal.
%First of all we find conditions under which $\mathtt{ZA}(R)$ is an empty graph.
%\begin{proposition}
%Let $R=\prod\limits_{i=1}^{n}R_i$ be a direct product of rings. 
%\begin{enumerate}
%\item If $R_i$'s are zero-dimensional, then for any nonzero nonunit element $x$ of $R$, ${\rm Ann}_R(x)\neq\{0\}$. 
%\item If for some $1\leq i\leq n$, $\mathtt{ZA}(R_i)$ is a nonempty graph, then so is $\mathtt{ZA}(R)$.
%\item If $R_i$'s are comparable with respect to inclusion and $\mathtt{ZA}(R)$ is a nonempty graph, then so is $\mathtt{ZA}(R_i)$
%for every $1\leq i\leq n$.
%\end{enumerate}
%\end{proposition}

\begin{theorem}\label{first}
Let $R$ be a ring. 
If $\mathtt{ZA}(R)$ is an empty graph, then $R$ is a local ring and $Ann_R(x)\neq\{0\}$
for every nonunit element $x\in R$. The converse is true if $R$ is a B\'{e}zout ring.
\end{theorem}
\begin{proof}
Assume that $\mathtt{ZA}(R)$ is empty. Let $\mathfrak{m}_1,\mathfrak{m}_2$ be two distinct maximal ideals of $R$.
Then $\mathfrak{m}_1+\mathfrak{m}_2=R$ implies that there exist $x\in \mathfrak{m}_1$ and $x_2\in \mathfrak{m}_2$ such that
$x+y=1$. So $x$ and $y$ are adjacent, which is a contradiction. Hence $R$ is a local ring.
Let $\mathfrak{m}$ be the maximal ideal of $R$ and $x$ be an element of $\mathfrak{m}$.
Suppose that ${\rm Ann}_R(x)=\{0\}$. Then $\{x^n|n\in\mathbb{N}\}$ is an infinite clique in $\mathtt{ZA}(R)$  that is a contradiction.
So ${\rm Ann}_R(x)\neq\{0\}$.\\
Suppose that $R$ is a local B\'{e}zout ring and $Ann_R(x)\neq\{0\}$
for every nonunit element $x\in R$. Let $x,y$ be two vertices in  $\mathtt{ZA}(R)$. Then $x,y\in\mathfrak{m}$. Hence 
$Rx+Ry=Rz$ for some nonzero nonunit element $z\in R$. So $x,y$ are not adjacent which shows that  $\mathtt{ZA}(R)$ is empty.
\end{proof}

\begin{remark}
Suppose that $R$ has a nontrivial idempotent element $e$. Then $e+(1-e)=1$ implies that $e$ and $1-e$
are adjacent. Hence ${\rm deg}_{\mathtt{ZA}(R)}(e)\geq1$ and so $\mathtt{ZA}(R)$ is not an empty graph.
\end{remark}

\begin{remark}
Let $R$ be a ring. 
Notice that if $R$ is an Artinian ring or a Boolean ring, then ${\rm dim}(R)=0$.
By \cite[Theorem 3.4]{AD}, ${\rm dim}(R)=0$ if and only if for every $x\in R$ there exists a positive integer $n$ such that $x^{n+1}$ divides $x^n$. Therefore, every nonzero nonunit element of a zero-dimensional ring has a nonzero annihilator.
 Hence, if $R$ is a zero-dimensional chained ring, then $\mathtt{ZA}(R)$ is an empty graph.
 
\end{remark}

Let ${\rm Z}^*(R)$ denote the zero divisors of $R$ and ${\rm Z}(R)=\rm{Z}^*(R)\cup\{0\}$. 
\begin{theorem}
Let $R$ be a ring and $S$ be a multiplicative closed subset of $R$ such that $S\cap {\rm Z}(R)=\{0\}$.
Then $\mathtt{ZA}(R)\simeq\mathtt{ZA}(R_S)$.
\end{theorem}
\begin{proof}
Define the vertex map $\Phi:{\rm V}(\mathtt{ZA}(R))\longrightarrow {\rm V}(\mathtt{ZA}(R_S))$ by $x\mapsto \frac{x}{1}$. 
We can easily verify that $x=y$ if and only if $\frac{x}{1}=\frac{y}{1}$. Also, it is easy to see that
${\rm Ann}_R(x)\cap {\rm Ann}_R(y)=\{0\}$ if and only if ${\rm Ann}_{R_S}(\frac{x}{1})\cap {\rm Ann}_{R_S}(\frac{y}{1})=\{\frac{0}{1}\}$.
\end{proof}

\begin{theorem}\label{finite}
Let $R$ be a B\'{e}zout ring with $|{ Max}(R)|<\infty$ such that $\delta(\mathtt{ZA}(R))>0$. Then
$\mathtt{ZA}(R)$ is a finite graph if and only if 
every vertex of $\mathtt{ZA}(R)$ has finite degree.
\end{theorem}
\begin{proof}
The ``only if'' part is evident.\\
Suppose that each vertex of $\mathtt{ZA}(R)$ has finite degree. If ${\rm Ann}_R(x)=\{0\}$ for
some nonzero nonunit element $x\in R$, then $x$ is adjacent to all vertices of $\mathtt{ZA}(R)$ 
that implies $\mathtt{ZA}(R)$ is a finite graph. Assume that ${\rm Ann}_R(x)\neq\{0\}$ for each nonzero nonunit element $x\in R$. 
We claim that ${\rm Jac}(R)=\{0\}$. On the contrary, assume that there exists a nonzero element $a\in {\rm Jac}(R)$. 
Since $\mathtt{ZA}(R)$ has no isolated vertex, $a$ is adjacent to another vertex, say $b$.
Since $R$ is a B\'{e}zout ring, $Ra+Rb$ is generated by a nonzero nonunit element $c$ of $R$
and so ${\rm Ann}_R(Ra+Rb)={\rm Ann}_R(c)\neq\{0\}$, which is impossible. So ${\rm Jac}(R)=\{0\}$. 
Hence by Chinese Remainder Theorem we have $R\simeq F_1 \times F_ 2\times\cdots\times F_n$, where 
$F_i$'s are fields and $n=|{\rm Max}(R)|$. 
Let $0\neq u\in F_1 $. Then $(u,0,\dots,0)$ and $(0,1,\dots,1)$ are adjacent.
Since $(0,1,\dots,1)$ has finite degree, so $F_1$ is a finite field.
Similarly we can show that $F_i$'s are finite fields.
Consequently $R$ has finitely many nonzero nonunit elements and the proof is
complete. 
\end{proof}

\begin{theorem}
Let $R$ be a B\'{e}zout ring with $|{ Max}(R)|<\infty$. 
Then the following conditions are equivalent:
\begin{enumerate}
\item $\mathtt{ZA}(R)$ is a bipartite graph with $\delta(\mathtt{ZA}(R))>0$;
\item $\mathtt{ZA}(R)$ is a complete bipartite graph;
\item $R\simeq F_1\times F_2$ where $F_1$ and $F_2$ are two fields.
\end{enumerate}
\end{theorem}
\begin{proof}
$(1)\Rightarrow(3)$
Suppose that $\mathtt{ZA}(R)$ is a bipartite graph with $\delta(\mathtt{ZA}(R))>0$.
If ${\rm Ann}_R(x)=\{0\}$ for some nonzero nonunit element $x$ of $R$, then
$\{x^n|n\in\mathbb{N}\}$ is an infinite clique that is a contradiction.
Then, for every nonzero nonunit element $x$ of $R$ we have ${\rm Ann}_R(x)\neq\{0\}$.
Similar to the proof of Theorem \ref{finite} we can show that 
$R=F_1 \times F_ 2\times\cdots\times F_n$, where 
$F_i$'s are fields and $n=|{\rm Max}(R)|$. Clearly $n\neq1$. 
If $n\geq3$, then
$\{(0,1,\dots,1),(1,0,1,\dots,1),(1,1,0,1,\dots,1)\}$
is a clique in $\mathtt{ZA}(R)$, a contradiction. So $R\simeq F_1\times F_2$.\\
$(3)\Rightarrow(2)$ Suppose that $R\simeq F_1\times F_2$ where $F_1$ and $F_2$ are two fields. 
Every vertex in $\mathsf{ZA}(R)$ is of the form $(u,0)$ or $(0,v)$ where $0\neq u\in F_1$
and $0\neq v\in F_2$. Also, two vertices 
$(u,0)$ and $(0,v)$ are adjacent. On the other hand, every two vertices $(u_1,0), (u_2,0)$  cannot be adjacent. \\
$(2)\Rightarrow(1)$ is clear.
\end{proof}

\begin{theorem}
Let $R$ be a ring and $n\geq2$ be a natural number. Then $$girth(\mathtt{ZA}(M_n(R)))=3.$$
\end{theorem}

\begin{proof}
For $n=2$, the following matrices are pairwise adjacent in  $\mathtt{ZA}({\rm M}_2(R))$:
$$\begin{pmatrix}
1 &0 \\
0 & 0
\end{pmatrix},
\begin{pmatrix}
 0&0\\
1 & 0
\end{pmatrix}
\mbox{ and }
\begin{pmatrix}
0 &1\\
0 &1
\end{pmatrix}.$$
For $n\geq3$, the following matrices are pairwise adjacent in  $\mathtt{ZA}({\rm M}_n(R))$:
$$\begin{pmatrix}
1 &0&0\dots 0 \\
0 & 1&0\dots 0\\
0 &0&0\dots 0\\
\vdots&\vdots&\ddots\\
0 &0&0\dots 0
\end{pmatrix},~~~~~~~~~~~~~
\begin{pmatrix}
1 &0&0&0\dots 0 \\
0 & 0&0&0\dots 0\\
0 & 0&1&0\dots 0\\
0 & 0&0&1\dots 0\\
\vdots&\vdots&\vdots&\ddots\\
0 &0&0&0\dots 1
\end{pmatrix}
$$
and
$$\begin{pmatrix}
0 &0&0&0\dots 0 \\
0 & 1&0&0\dots 0\\
0 & 0&1&0\dots 0\\
0 & 0&0&1\dots 0\\
\vdots&\vdots&\vdots&\ddots\\
0 &0&0&0\dots 1
\end{pmatrix}.$$
\end{proof}

\section{When is $\mathtt{ZA}(R)$ connected?}

A ring $R$ is called {\it semiprimitive} if ${\rm Jac}(R)=0$, \cite{lam1}.
A ring $R$ is semiprimitive if and only if it is a subdirect product of fields, \cite[p. 179]{lam2}.
\begin{theorem}\label{conn1}
Let $R$ be a semiprimitive ring. If at least one of the maximal ideals of $R$ is principal, then
$\mathtt{ZA}(R)$ is a connected graph with ${\rm diam}(\mathtt{ZA}(R))\leq4$.
\end{theorem}
\begin{proof}
Suppose that $\mathfrak{m}$ is a maximal ideal of $R$ where $\mathfrak{m}=Rt$
for some $t\in R$. Let $x,y$ be two different nonzero nonunit elements of $R$. Consider the following cases:\\
{\bf Case 1.} Let $x,y\notin\mathfrak{m}$. Then $Rx+\mathfrak{m}=R$ and $Ry+\mathfrak{m}=R$.
Hence $x, y$ are adjacent to $t$. So $d_{\mathtt{ZA}(R)}(x,y)\leq2$.\\
{\bf Case 2.} Let $x\in\mathfrak{m}$ and $y\notin\mathfrak{m}$. 
Notice that $y$ is adjacent to $t$.
Since ${\rm Jac}(R)=\{0\}$,
there exists a maximal ideal $\mathfrak{m}^{\prime}$ different from $\mathfrak{m}$
such that $x\notin\mathfrak{m}^{\prime}$. 
So $Rx+\mathfrak{m}^{\prime}=R$, and thus there exist elements $r\in R$ and $z\in\mathfrak{m}^{\prime}$
such that $rx+z=1$. 
Therefore ${\rm Ann}_R(x)\cap {\rm Ann}_R(z)=\{0\}$.
So $x$ is adjacent to $z$.
Clearly $z\notin\mathfrak{m}$. Then $z$ is adjacent to $t$. Hence $d_{\mathtt{ZA}(R)}(x,y)\leq3$.\\
{\bf Case 3.} Let $x,y\in\mathfrak{m}$. A manner similar to Case 2 shows that
$d_{\mathtt{ZA}(R)}(x,t)\leq2$ and $d_{\mathtt{ZA}(R)}(y,t)\leq2$. Therefore
$d_{\mathtt{ZA}(R)}(x,y)\leq4$.\\
Consequently $\mathtt{ZA}(R)$ is a connected graph with ${\rm diam}(\mathtt{ZA}(R))\leq4$.
\end{proof}

\begin{theorem}\label{conn2}
Let $R$ be a B\'{e}zout ring. If $\mathtt{ZA}(R)$ is connected, then one of the following conditions holds:
\begin{enumerate}
\item There exists a nonzero nonunit element $x$ of $R$ such that $Ann_R(x)=\{0\}$,
\item $Jac(R)=\{0\}$,
\item $Jac(R)=\{0,x\}$ where $x$ is the only nonzero nonunit element of $R$.
\end{enumerate}
\end{theorem}
\begin{proof}
Assume that for every nonzero nonunit element $x$ of $R$, ${\rm Ann}_R(x)\neq\{0\}$ and also ${\rm Jac}(R)\neq\{0\}$.
Let $x$ be a nonzero element in ${\rm Jac}(R)$. Suppose that $\mathtt{ZA}(R)$ has a vertex $y$ different from $x$. Thus
 $Rx+Ry=Rz$ for some $z\in R$, because $R$ is a B\'{e}zout ring. Notice that $y\in \mathfrak{m}$ 
for some maximal ideal $\mathfrak{m}$ of $R$. Hence $z$ is nonzero nonunit and so by assumption ${\rm Ann}_R(z)\neq\{0\}$,
which shows that $x$ and $y$ are not adjacent. This contradiction implies that  $|V(\mathtt{ZA}(R))|=1$,
and so ${\rm Jac}(R)=\{0,x\}$.
\end{proof}

As a direct consequence of Theorem \ref{conn1} and Theorem \ref{conn2} we have the following result.
\begin{corollary}
Let $R$ be a B\'{e}zout ring 
 such that at least one of the maximal ideals of $R$ is principal. 
Then $\mathtt{ZA}(R)$ is connected if and only if one of the following conditions holds:
\begin{enumerate}
\item There exists a nonzero nonunit element $x$ of $R$ such that $Ann_R(x)=\{0\}$,
\item $Jac(R)=\{0\}$,
\item $Jac(R)=\{0,x\}$ where $x$ is the only nonzero nonunit element of $R$.
\end{enumerate}
\end{corollary}

\begin{theorem}
Let $R=F_1\times F_2\times\cdots\times F_n$ where $F_i$'s are fields. Then
$\mathtt{ZA}(R)$ is a connected graph with
$$
{\rm diam}(\mathtt{ZA}(R))=
\begin{cases}
1 &\mbox{ if } n=2 \mbox{ and } |F_1|=|F_2|=2 \\ 
2 &\mbox{ if } n=2 \mbox{ and either } |F_1|>2 \mbox{ or } |F_2|>2\\
3 &\mbox{ if } n \geq3.
\end{cases}
$$
\end{theorem}
\begin{proof}
Let $n=2$. In this case every vertex in $\mathsf{ZA}(R)$ is of the form $(u,0)$ or $(0,v)$ where $u\neq0$
and $v\neq0$. Furthermore, two vertices 
$(u,0)$ and $(0,v)$ are adjacent. \\
In the case when $n=2$ and $|F_1|=|F_2|=2$, we have $R\simeq\mathbb{Z}_2\times\mathbb{Z}_2$. So $\mathsf{ZA}(R)\simeq K_2$.\\
Let $n=2$ and $|F_1|>2$. In this case, 
every two different vertices $(u_1,0)$ and $ (u_2,0)$  cannot be adjacent. 
On the other hand $(u_1,0)$ and $ (u_2,0)$ are adjacent to $(0,1)$.
So ${\rm d}_{\mathsf{ZA}(R)}((u_1,0), (u_2,0))=2$.
Hence ${\rm diam}(\mathtt{ZA}(R))=2$.\\
Now, let $n\geq3$. Assume that $u=(u_1,u_2,\dots,u_n)$ and $v=(v_1,v_2,\dots,v_n)$ are two different vertices.
There exist two indexes $i,j$ such that $u_i\neq0$ and $v_j\neq0$. 
So $u=(u_1,u_2,\dots,u_n)$ is adjacent to $(1,\dots,1,\overbrace{0}^{i-\mbox{th}},1,\dots,1)$.
Also $v=(v_1,v_2,\dots,v_n)$ is adjacent to $(1,\dots,1,\overbrace{0}^{j-\mbox{th}},1,\dots,1)$.
Furthermore, if $i\neq j$, then the vertex $(1,\dots,1,\overbrace{0}^{i-\mbox{th}},1,\dots,1)$ is adjacent to
$(1,\dots,1,\overbrace{0}^{j-\mbox{th}},1,\dots,1)$.
Thus $\mathtt{ZA}(R)$ is connected and $d_{\mathtt{ZA}(R)}(u,v)\leq3$.
In special case, we have the following  path
$$(0,1,0,\dots,0)-(1,0,1,\dots,1)-(0,1,\dots,1)-(1,0,\dots,0).$$ Consequently ${\rm diam}(\mathtt{ZA}(R))=3$.
\end{proof}

\section{When is $\mathtt{ZA}(R)$ star?}

\begin{lemma}\label{star1}
Let $R$ be a ring. If $\mathtt{ZA}(R)$ is a star, then $|{Max}(R)|\leq 2$. 
\end{lemma}
\begin{proof}
Suppose that $\mathtt{ZA}(R)$ is a star.
If $\mathfrak{m}$ and $\mathfrak{m}^{\prime}$ is two different maximal ideals of $R$, then for every 
$x\in\mathfrak{m}\backslash\mathfrak{m}^{\prime}$ we have
$Rx+\mathfrak{m}^{\prime}=R$. Hence there exist elements $r\in R$ and $y\in\mathfrak{m}^{\prime}\backslash\mathfrak{m}$ 
such that $rs+y=1$.
Therefore ${\rm Ann}_R(x)\cap {\rm Ann}_R(y)=\{0\}$. So $x$ and $y$ are adjacent.
Let $\mathfrak{m}_1, \mathfrak{m}_2$ and $\mathfrak{m}_3$ be three different maximal ideals of $R$.
Then there are elements $a\in\mathfrak{m}_1\backslash(\mathfrak{m}_2\cup
\mathfrak{m}_3)$,
$b\in\mathfrak{m}_2\backslash(\mathfrak{m}_1\cup\mathfrak{m}_3)$ and $c\in\mathfrak{m}_3\backslash (\mathfrak{m}_1\cup\mathfrak{m}_2)$.
Then either $a,b,c$ are pairwise adjacent or there exist at least two disjoint edges in $\mathtt{ZA}(R)$, which is a contradiction.
Consequently  $|{\rm Max}(R)|\leq 2$. 
\end{proof}

\begin{theorem}\label{star}
Let $R$ be a B\'{e}zout ring that is not a field. Then $\mathtt{ZA}(R)$ is a star if and only if one of the following conditions holds:
\begin{enumerate}
\item $(R,\mathfrak{m})$ when $\mathfrak{m}=\{0,x\}$ in which $x$ is a nonzero element of $R$ with $x^2=0$,
\item $R\simeq\mathbb{Z}_2\times F$ where $F$ is a field. 
\end{enumerate}
\end{theorem}
\begin{proof}
$(\Rightarrow)$ Suppose that $\mathtt{ZA}(R)$ is a star.
Hence $|{\rm Max}(R)|\leq 2$, by Lemma \ref{star1}. 
Notice that if ${\rm Ann}_R(t)=\{0\}$ for some element $t$ of a maximal ideal $\mathfrak{m}$, then
$\{t^n|n\in\mathbb{N}\}$ is an infinite clique that is impossible.
Consider the following cases:\\
{\bf Case 1.} ${\rm Max}(R)=\{\mathfrak{m}\}$. 
Let $x$ be a nonzero element in $\mathfrak{m}$.
Then by Theorem \ref{first}, $\mathtt{ZA}(R)$ is empty and so
$\mathfrak{m}=\{0,x\}$. On the other hand, by Nakayama's Lemma we have that $x^2=0$.\\
{\bf Case 2.} ${\rm Max}(R)=\{\mathfrak{m}_1,\mathfrak{m}_2\}$. Since $\mathfrak{m}_1+\mathfrak{m}_2=R$,
there exist $x\in\mathfrak{m}_1$ and $y\in\mathfrak{m}_2$ such that $x+y=1$. Hence $x$ and $y$ are adjacent. Now, if there exists 
$0\neq z\in\mathfrak{m}_1\cap\mathfrak{m}_2$, then 
$z$ is not adjacent to $x$ and $y$, because $R$ is 
a B\'{e}zout ring and ${\rm Ann}_R(t)=\{0\}$ for every nonzero nonunit element $t$ of $R$. 
This contradiction shows that $\mathfrak{m}_1\cap\mathfrak{m}_2=\{0\}$.
Hence by Chinese Remainder Theorem we deduce that $R\simeq R/\mathfrak{m}_1\oplus R/\mathfrak{m}_2$.
If there exist nozero elements $a_1,a_2\in R/\mathfrak{m}_1$ and $b_1,b_2\in R/\mathfrak{m}_2$, then
we have the following path $$(a_1,0)-(0,b_1)-(a_2,0)-(0,b_2),$$ a contradiction. Hence we can assume that $R/\mathfrak{m}_1=\mathbb{Z}_2$.\\
$(\Leftarrow)$
If (1) holds, the clearly $\mathtt{ZA}(R)$ is a star. Assume that (2) holds.
Notice that $(1,0)$ is adjacent to all vertices $(0,u)$ where $u$ is a nonzero element of $F$.
Also, for every two different elements $u_1,u_2\in F$, $(0,u_1)$ and $(0,u_2)$ are not adjacent.
Consequently $\mathtt{ZA}(R)$ is a star.
\end{proof}

\section{When is $\mathtt{ZA}(R)$ complete?}

\begin{proposition}\label{com}
Let $R$ be a ring. If $\mathtt{ZA}(R)$ is a complete graph, then $\mathcal{A}_R$ is a complete graph.
\end{proposition}
\begin{proof}
 Assume that $\mathtt{ZA}(R)$ is a complete graph.
Let $I,J$ be two nonzero proper ideals of $R$. Then there are two different nonzero nonunit elements $x,y\in R$
such that $x\in I$ and $y\in J$. Hence ${\rm Ann}_R(I)\cap {\rm Ann}_R(J)\subseteq {\rm Ann}_R(x)\cap {\rm Ann}_R(y)=\{0\}$.
Therefore $I$ and $J$ are adjacent.
\end{proof}

The following remark shows that the converse of Proposition \ref{com} is not true.
\begin{remark}
Consider the ring $R=\mathbb{Z}_5\times\mathbb{Z}_5$. By \cite[Theorem 6]{AAAS}, $\mathcal{A}_R(=K_2)$ is a complete graph.
But $\mathtt{ZA}(R)$ is a 4-regular graph that is not a complete graph.
$$
\begin{tikzpicture}
\vertex (a1) [fill] at (0,0) [label=left:{(1,0)}]{};
\vertex (a2) [fill]at (0,-0.8) [label=left:{(2,0)}]{};
\vertex (a3) [fill]at (0,-1.6) [label=left:{(3,0)}]{};
\vertex (a4) [fill]at (0,-2.4) [label=left:{(4,0)}]{};
\vertex (b1) [fill]at (0.8,0) [label=right:{(0,1)}]{};
\vertex (b2) [fill]at (0.8,-0.8) [label=right:{(0,2)}]{};
\vertex (b3) [fill]at (0.8,-1.6) [label=right:{(0,3)}]{};
\vertex (b4) [fill]at (0.8,-2.4) [label=right:{(0,4)}]{};
\path
(a1) edge (b1)
(a1) edge (b2)
(a1) edge (b3)
(a1) edge (b4)
(a2) edge (b1)
(a2) edge (b2)
(a2) edge (b3)
(a2) edge (b4)
(a3) edge (b1)
(a3) edge (b2)
(a3) edge (b3)
(a3) edge (b4)
(a4) edge (b1)
(a4) edge (b2)
(a4) edge (b3)
(a4) edge (b4);
\end{tikzpicture}
$$
\begin{center}
Fig. 1. $\mathtt{ZA}(R)$
\end{center}
\end{remark}

\begin{theorem}
Let $R$ be a ring. Then $\mathtt{ZA}(R)$ is a complete
graph if and only if one of the following conditions holds:
\begin{enumerate}
\item $R$ has exactly one nonzero nonunit element,
\item $R$ is an integral domain,
\item $R=\mathbb{Z}_2\times\mathbb{Z}_2$.
\end{enumerate}
\end{theorem}
\begin{proof}
$(\Rightarrow)$ Assume that $\mathtt{ZA}(R)$ is a complete graph. Then, by Proposition \ref{com},
$\mathcal{A}_R$ is a complete graph. Suppose that $R$ is not an integral domain. So there exists 
a nonzero nonunit element $x\in R$ such that ${\rm Ann}_R(x)\neq\{0\}$.
Therefore, \cite[Theorem 6]{AAAS} implies that
either $R$ has exactly one nonzero proper ideal or $R$ is a direct product of two fields. 
Suppose that the former case holds. If $y$ is a nonzero nonunit element of $R$
different from $x$, then $Rx=Ry$. So ${\rm Ann}_R(x)\cap {\rm Ann}_R(y)={\rm Ann}_R(x)\neq\{0\}$,
which is a contradiction. Therefore $R$ has exactly one nonzero nonunit element.
Now, let $R$ be a direct product of two fields, say $R= F_1\times F_2$. 
If there exist two different nonzero elements $u,v$ in $F_1$, then $(u,0)$ and $(v,0)$ cannot be
adjacent. Hence $F_1=\mathbb{Z}_2$. Similarly, we can show that $F_2=\mathbb{Z}_2$.
Consequently $R= \mathbb{Z}_2\times \mathbb{Z}_2$.\\
$(\Leftarrow)$ Clearly, if (1) or (2) holds, then  $\mathtt{ZA}(R)$ is a complete
graph. Assume that (3) holds. Then $\mathtt{ZA}(R)\simeq K_2$ and we are done.
\end{proof}

\section{When is $\mathtt{ZA}(R)$  $k$-regular?}

Recall that a finite field of order $q$ exists if and only if the order $q$ is a prime power $p^s$.
A finite field of order $p^s$ is denoted by $\mathbb{F}_{p^s}$.
\begin{theorem}
Let $R$ be a B\'{e}zout ring with $|{Max}(R)|<\infty$.
Then $\mathtt{ZA}(R)$ is a $k$-regular graph $(0<k<\infty)$ if and only if  $R\simeq\mathbb{F}_{k+1}\times\mathbb{F}_{k+1}$.
\end{theorem}
\begin{proof}
The `` if " part has a routine verification. Let $\mathtt{ZA}(R)$ be a $k$-regular graph $(0<k<\infty)$.
If ${\rm Ann}_R(x)=\{0\}$ for some nonzero nonunit element $x$ of $R$, then
$\{x^n|n\in\mathbb{N}\}$ is an infinite clique that is a contradiction.
Then, for every nonzero nonunit element $x$ of $R$ we have ${\rm Ann}_R(x)\neq\{0\}$.
Similar to the manner that described in the proof
of Theorem \ref{finite}, we have $R\simeq F_1\times F_2\times\cdots\times F_n$ where $F_i$'s are fields
and $n=|{\rm Max}(R)|$. Since
${\rm Ann}_R((1,0,\dots,0))=0\times F_2\times F_3\times\cdots\times F_n$
and
${\rm Ann}_R((0,1,0,\dots,0))=F_1\times 0\times F_3\times\cdots\times F_n$, 
then
$${\rm N}_{\mathtt{ZA}(R)}((1,0,\dots,0))=\{(0,u_2,\dots,u_n)|u_i\in F_i\backslash\{0\} \mbox{ for } 2\leq i\leq n\}$$ 
and
$${\rm N}_{\mathtt{ZA}(R)}((0,1,0,\dots,0))=\{(u_1,0,u_3,\dots,u_n)|u_i\in F_i\backslash\{0\} \mbox{ for } 1\leq i\leq n,~ i\neq2\}.$$ 
So
$$(|F_2|-1)(|F_3|-1)\cdots(|F_n|-1)=(|F_1|-1)(|F_3|-1)\cdots(|F_n|-1),$$
because $\mathtt{ZA}(R)$ is $k$-regular.
Hence $|F_1|=|F_2|$.
Similarly we can show that $|F_1|=|F_2|=\dots=|F_n|$.
Let $n\geq3$.
Note that ${\rm N}_{\mathtt{ZA}(R)}((1,1,0,\dots,0))$ is the union of the following sets
$$\{(u_1,0,u_3,\dots,u_n)|u_i\in F_i\backslash\{0\} \mbox{ for } 1\leq i\leq n, ~ i\neq2\},$$
$$\{(0,u_2,\dots,u_n)|u_i\in F_i\backslash\{0\} \mbox{ for } 2\leq i\leq n\}$$
and
$$\{(0,0,u_3,\dots,u_n)|u_i\in F_i\backslash\{0\} \mbox{ for } 3\leq i\leq n\}.$$
Therefore
$$(|F_1|-1)^{n-1}=2(|F_1|-1)^{n-1}+(|F_1|-1)^{n-2},$$
since $\mathtt{ZA}(R)$ is $k$-regular.
Thus $|F_1|=0$ which is a contradiction.
Consequently $n=2$. 
If there exist two different nonzero elements $u,u^{\prime}$ in $F_1$, then $(u,0)$ and $(u^{\prime},0)$ cannot be
adjacent. On the other hand for every nonzero elements $u\in F_1$ and $v\in F_2$, $(u,0)$ and $(0,v)$ are adjacent.
So ${\rm deg}_{\mathtt{ZA}(R)}((u,0))=|F_1|-1=k$. Therefore $R\simeq\mathbb{F}_{k+1}\times\mathbb{F}_{k+1}$.
\end{proof}

\begin{corollary}
Let $R$ be a B\'{e}zout ring with $|{\rm Max}(R)|<\infty$.
If $\mathtt{ZA}(R)$ is a $k$-regular graph $(0<k<\infty)$, then $k+1$ is a prime power. 
\end{corollary}

\section{Chromatic number and clique number of $\mathtt{ZA}(R)$}

%\begin{remark}
%Consider the ring $\mathbb{Z}_2\times\mathbb{Z}_4$
%$$
%\begin{tikzpicture}
%\vertex (a2) [fill]at (0,-0.8) [label=left:{(1,0)}]{};
%\vertex (b1) [fill]at (0.8,0) [label=right:{(0,1)}]{};
%\vertex (b3) [fill]at (0.8,-1.6) [label=right:{(0,3)}]{};
%\path
%(a2) edge (b1)
%(a2) edge (b3);
%\end{tikzpicture}
%\hspace{2cm}
%\begin{tikzpicture}
%\vertex (a2) [fill]at (0,0) [label=left:{\mathbb{Z}}]{};
%\vertex (b1) [fill]at (0.8,0) [label=right:{}]{};
%\path
%(a2) edge (b1);
%\end{tikzpicture}
%$$
%\begin{center}
%Fig. 2. $\mathtt{ZA}(R)$ \hspace{2.8cm} Fig. 3. $\mathcal{A}_R$
%\end{center}
%\end{remark}

Recall that, a ring $R$ is said to be {\it reduced} if it has no nonzero
nilpotent elements. 
\begin{theorem}
If $R$ is a reduced Noetherian ring, then the chromatic number of $\mathtt{ZA}(R)$ is
infinite or $R$ is a direct product of finitely many fields. 
\end{theorem}
\begin{proof}
The proof is similar to that of \cite[Theorem 16]{AAAS}.
\end{proof}

\begin{lemma}\label{cap}

Let $P_1$ and $P_2$ be two prime ideals of a ring $R$ with $P_1\cap P_2=\{0\}$. Then every two 
nonzero elements $x\in P_1$ and $y\in P_2$ are adjacent.
\end{lemma}
\begin{proof}
Suppose that $r\in {\rm Ann}_R(x)\cap {\rm Ann}_R(y)$. Since $rx=0\in P_2$ and $x\notin P_2$, then $r\in P_2$.
Similarly it turns out that $r\in P_1$. Hence $r\in P_1\cap P_2=\{0\}$.
\end{proof}

\begin{theorem}\label{clique}
Let $R$ be a ring and $n\geq2$ be a natural number. If either $|Min(R)|=n$ or $R=R_1\times R_2\times\cdots\times R_n$ where $R_i$'s are rings, then
$\omega(\mathtt{ZA}(R))\geq n$.
\end{theorem}
\begin{proof}
 Assume that 
${\rm Min}(R)=\{\mathfrak{p}_1,\mathfrak{p}_2,\dots,\mathfrak{p}_n\}$ where $\mathfrak{p}_i$'s are nonzero.
So, by Lemma \ref{cap}, $n\leq\omega(\mathtt{ZA}(R))$.
Now, suppose that $R=R_1\times R_2\times\cdots\times R_n$ where $R_i$'s are rings.
Then $\{(1,\dots,1,\overbrace{0}^{i-\mbox{th}},1,\dots,1)|1\leq i\leq n\}$ is a clique in $\mathtt{ZA}(R)$ and the result follows.
\end{proof}

\vspace{5mm} \noindent \footnotesize 
\begin{minipage}[b]{10cm}
Hojjat Mostafanasab \\
Department of Mathematics and Applications, \\ 
University of Mohaghegh Ardabili, \\ 
P. O. Box 179, Ardabil, Iran. \\
Email: h.mostafanasab@gmail.com
\end{minipage}\\

\end{document}